\documentclass[11pt]{amsart}

\usepackage{amsfonts,amsmath,amssymb,amsthm}

\usepackage[svgnames,hyperref]{xcolor}

\newcommand{\nicecolor}{Navy}
\usepackage[shortlabels]{enumitem}
\setlist[1]{wide}
\setlist[2]{leftmargin=15mm}
\setlist[enumerate]{label=\rm{(\arabic*)}}
\setlist[enumerate,2]{label=\rm({\it\roman*}), }
\setlist[itemize]{label=\raisebox{0.25ex}{\tiny$\bullet$}}
 
\usepackage[backref, colorlinks, linktocpage, citecolor = \nicecolor, linkcolor = \nicecolor]{hyperref} 

\newtheorem{theorem}{Theorem}

\newtheorem{lemma}{Lemma}[section]
\newtheorem{corollary}[lemma]{Corollary}
\newtheorem{proposition}[lemma]{Proposition}

\newtheorem{conjecture}[lemma]{Conjecture}

\theoremstyle{definition}
\newtheorem{definition}[lemma]{Definition}

\theoremstyle{remark}
\newtheorem{remark}[lemma]{Remark}

\newtheorem{example}[lemma]{Example}

\newcommand\kk{{\mathbf k}}
\renewcommand\k{\kk}

\newcommand\A{{\mathbb A}}
\newcommand\iso{\stackrel{\simeq}{\longrightarrow}}
\newcommand{\Aut}{\mathrm{Aut}}
\newcommand{\Jac}{\mathrm{Jac}}

\newcommand{\Spec}{\mathrm{Spec}}

\title{Bivariables and V\'en\'ereau polynomials}
\author[J.~Blanc]{J\'er\'emy~Blanc}
\address{J. Blanc, Universit\"{a}t Basel, Departement Mathematik und Informatik, Spiegelgasse $1$, CH-$4051$ Basel, Switzerland.}
\email{jeremy.blanc@unibas.ch}
\urladdr{http://algebra.dmi.unibas.ch/blanc/}

\author[P.-M.~Poloni]{Pierre-Marie~Poloni}
\address{P.-M.~Poloni, Universit\"{a}t Basel, Departement Mathematik und Informatik, Spiegelgasse $1$, CH-$4051$ Basel, Switzerland.}
\email{pierre-marie.poloni@unibas.ch}

\thanks{The authors acknowledge support by the Swiss National Science Foundation Grant \textquotedblleft Curves in the spaces\textquotedblright 200021--169508.}

\begin{document}

\begin{abstract}
We study a family of polynomials introduced by Daigle and Freudenburg, which contains the famous V\'en\'ereau polynomials and  defines $\mathbb{A}^2$-fibrations over $\mathbb{A}^2$. According to the Dolgachev-Weisfeiler conjecture, every such fibration  should have the structure of a locally trivial $\mathbb{A}^2$-bundle over $\mathbb{A}^2$.  We follow an idea of Kaliman and Zaidenberg   to show that  these fibrations are locally trivial $\mathbb{A}^2$-bundles over the punctured plane, all of the same specific form $X_f$, depending on an element $f\in k[a^{\pm 1},b^{\pm 1}][x]$. We then introduce the notion of bivariables and show that the set of bivariables is   in bijection with the set of locally trivial bundles $X_f$ that are trivial. This allows us to give another proof of Lewis's result stating that the second V\'en\'ereau polynomial is a variable  and also to trivialise other elements of the family $X_f$. We hope that the terminology and methods developed here may lead to future study of the whole family $X_f$.
\end{abstract}

\maketitle

\section{Introduction} 

Throughout this paper, we work over a fixed ground field $\kk$ and all algebraic varieties and morphisms are defined over it.

The Dolgachev-Weisfeiler conjecture \cite[Conjecture 3.8.5]{DW}   is a famous open problem in affine algebraic geometry that concerns \emph{$\A^n$-fibrations}, i.e.~morphisms $X\to Y$ between affine varieties with the property that every fibre is isomorphic to the $n$-dimensional affine space $\A^n$.
 
The conjecture predicts that every such $\A^n$-fibration should have the structure of an $\A^n$-bundle (locally trivial in the Zariski topology), when the target variety $Y$ is normal. 

Recall that every $\A^n$-bundle over an affine variety is isomorphic to a vector bundle (Bass-Connell-Wright Theorem \cite{BCW}) and moreover that every vector bundle over $\A^m$ is trivial (Quillen-Suslin Theorem \cite{Qui,Sus}). Hence, the Dolgachev-Weisfeiler conjecture is often reformulated as follows in the case where the target variety is an affine space. 

\begin{conjecture}[Dolgachev-Weisfeiler conjecture] 
 Every $\A^n$-fibration $X\to \A^m$ is isomorphic to the trivial fibration $\A^m\times \A^n\to \A^m$.
\end{conjecture}

This was proven to be true for $(n,m)=(1,1)$ in \cite[Proposition 3.7]{DW}, when $n=1$ and $m$ is arbitrary in \cite{KW1985}, and when $(n,m)=(2,1)$ in \cite{KZ2001}. See also \cite{KR2014} for a recent proof of these results. All other cases remain wide open. 

In this paper, we focus on the case $n=m=2$ and study  a family of $\A^2$-fibrations introduced by Daigle and Freudenburg in \cite{DF}. These fibrations are of the form 
\[\pi_{P,n}\colon\A^4\to\A^2, (x,y,z,u)\mapsto(x,v_{P,n}),\]
 where $n$ denotes a positive integer, $P(z)\in\kk[z]$ a polynomial of degree at least two and where $v_{P,n}\in \kk[x,y,z,u]$ is given by
\[v_{P,n}=y+x^n\left(xz+y(yu+P(z))\right).\]
The special case where $n=1$ and $P(z)=z^2+z$ corresponds to an old example due to Bhatwadekar and Dutta \cite{BD} whereas the polynomials $v_{z^2,n}$ are the famous    polynomials introduced in the PhD thesis of V\'en\'ereau \cite{Ven}.

Let us  recall the notion of \emph{$A$-variable} of a ring $B$. We say that  an element $f\in B$ is an $A$-variable of $B$, if $B=A[X_1,\ldots, X_k]$ is a polynomial ring over the commutative ring  $A$ and if there exits an automorphism of $B$ that fixes $A$ and sends $f$ onto one of the indeterminates. With this notion, we can reformulate  the above problem  in algebraic terms as follows. The $\A^2$-fibration $\pi_{P,n}$ is a (locally) trivial $\A^2$-bundle if and only $v_{P,n}$ is a $\k[x]$-variable of $\k[x,y,z,u]$.

In his PhD Thesis, V\'en\'ereau showed that $v_{z^2,n}$ is a $\k[x]$-variable of the ring $\kk[x,y,z,u]=\kk[x][y,z,u]$ for each $n\geq3$. More recently, Lewis \cite{Lewis} succeeded to prove that $v_{z^2,n}$ is also a $\k[x]$-variable. Nevertheless, it is still unknown whether $v_{z^2,1}$ is a $\k[x]$-variable, or even if it is a $\k$-variable of $\kk[x,y,z,u]$. Note that the latter question, which is weaker a priori, corresponds to the Dolgachev-Weisfeiler conjecture applied to the $\A^3$-fibration given by $\A^4\to \A^1$, $(x,y,z,u)\mapsto v_{z^2,1}$. 
On the other hand, one can prove that every polynomial $v_{P,n}$ is $1$-stable $\k[x]$-variable (see Proposition~\ref{Prop:variables_stables}).

Let us fix from now on some coordinates $a,b$ on $\A^2$. In \cite{KZ2004}, Kaliman and Zaidenberg introduced an interesting strategy to study V\'en\'ereau polynomials based on the fact that the morphisms $\pi_{z^2,n}$ are trivial $\A^2$-bundles over the two open subsets $U_a=\A^2\setminus\{a=0\}$ and $U_b=\A^2\setminus\{b=0\}$. Using this idea, they could reprove that $\pi_{z^2,n}$ is isomorphic to the trivial fibration for every $n\ge 3$. Nevertheless, they couldn't push this technique further at that time and were not able to treat the case of $v_{z^2,n}$ with $n=1$ or $n=2$. 
In the present paper, we give Kaliman and Zaidenberg's strategy another try and want to apply it to the more general polynomials $v_{P,n}$. For this, we will introduce the notion bivariables. The fact that the second V\'en\'ereau polynomial $v_{z^2,2}$, and more generally all $v_{P,2}$ with $P$ of degree $2$, are $\k[x]$-variable will follow quite simply from this new point of view.

We shall indeed prove  in Section~\ref{sec:localtrivialityofpiPn} that every map $\pi_{P,n}$ is a trivial $\A^2$-bundle over $U_a$ and over $U_b$. The preimage of the origin being also isomorphic to $\A^2$ (since it is given by the equations $x=y=0$), this implies that every $\pi_{P,n}$ is an $\A^2$-fibration and therefore that it should have, according to the Dolgachev-Weisfeiler conjecture, the structure of  an $\A^2$-bundle. To check that this is the case, one needs either to find a neighbourhood of the origin over which the fibration is trivial, or to show that $\pi_{P,n}$ defines the trivial fibration over the punctured plane $\A^2_{*}=\A^2\setminus \{(0,0)\}=U_a\cup U_b$ (see Lemma~\ref{Lem:Equivalenttrivialities}).

Having this question in mind, we will consider the following family of locally trivial $\A^2$-bundles over $\A^2_{*}$, obtained by gluing trivial bundles over $U_a$ and $U_b$ along their intersection $U_{ab}=U_a\cap U_b=\A^2\setminus \{ab=0\}$.
\begin{definition}\label{Defi:Xf} Given an element $f=f(a,b,x)\in \kk[a^{\pm 1},b^{\pm 1}][x]$, we denote by $X_f$ the variety obtained by gluing together the affine varieties
$U_a \times \A^2$ and $U_{b}\times \A^2$ by means of the transition function 
\[\begin{array}{ccc}
U_{ab}\times \A^2& \to& U_{ab}\times \A^2\\
((a,b),(x,y))&\mapsto& \left((a,b),(x,y+f(a,b,x))\right).\end{array}\]
We denote by $\rho_f\colon X_f\to \A^2_*$ the $\A^2$-bundle over the punctured affine plane $\A^2_*=U_a\cup U_b$ given by the projection onto the first factor.
\end{definition}

In Section~\ref{sec:localtrivialityofpiPn}, we will prove the following result, which shows that one can reduce the study of the morphisms $\pi_{P,n}$ to the study of a special family of varieties $X_f$.
\begin{theorem}\label{Theorem:fcs_de_transition}
For each $n\ge 1$ and each $P\in \k[z]$, the restriction of the $\A^2$-fibration $\pi_{P,n}$ over the punctured plane $\A^2_*$ is a locally trivial $\A^2$-bundle isomorphic to $\rho_f\colon X_f\to \A^2_*$ with 
\[f=\frac{x}{ab^2}-\frac{1}{ab^m}\cdot \frac{b^m-(a^nx)^m}{b-a^nx}P(\frac{x}{a})\in \kk[a^{\pm1},b^{\pm1}][x],\]
where $m$ is any integer such that $mn>\deg(P)$.
\end{theorem}
In the special case where $P=x^2$, if we  choose the smallest $m$ such that $mn>2$, we get  that the fibration associated to the $n$-th V\'enereau polynomial is isomorphic to the fibration $\rho_{f_n}$, where
\[
f_n=\left \{ \begin{array}{lll}
\frac{x}{ab^2}-\frac{x^2}{a^3b} & \text{when}&n\ge 3,\vspace{0.1cm}\\
\frac{x}{ab^2}-\frac{x^2}{a^3b}-\frac{x^3}{ab^2} & \text{when}&n=2,\vspace{0.1cm}\\
\frac{x}{ab^2}-\frac{x^2}{a^3b}-\frac{x^3}{a^2b^2}-\frac{x^4}{ab^3}& \text{when}& n=1\end{array}\right.\]
(see Example~\ref{Exa:FforVenereau}).
We recover here at once the formulas that were first computed in \cite[Proposition 2]{KZ2004}.\\

As the transition functions of the $\A^2$-bundle $\rho_f\colon X_f\to \A^2_*$    fix the  coordinate $x$, the variety $X_f$ can be also  naturally seen as an $\A^1$-bundle $X_f\to \A^2_{*}\times \A^1$ via the projection onto the first three coordinates. On the contrary to the difficult question of the isomorphism as $\A^2$-bundles, it is  straightforward to decide whether two varieties $X_f$ and $X_g$ are isomorphic as $\A^1$-bundles (see  Lemma~\ref{Lemm:A1bundle}). As we will explain in Section~\ref{Sec:Bivariables}, the triviality of an $\A^2$-bundle  $\rho_f\colon X_f\to \A^2_*$ is related to the notion of \emph{bivariables} which we introduce in Definition~\ref{Defi:bivariable}. A bivariable is a polynomial in $\k[a,b][x,y]$ that becomes a variable when being seen as an element of $\k[a^{\pm1},b][x,y]$ as well as when   seen  as an element of $\k[a,b^{\pm1}][x,y]$. Since the group $\Aut_{\k[a,b]}(\kk[a,b][x,y])$ naturally acts on the set of bivariables, we may consider bivariables up to the action of this group (see Definition~\ref{Rem:ActionAutOnBivariables}). This leads us to a natural correspondence between bivariables and $\A^2$-bundles $X_f$ that are trivial. More precisely, we shall establish the following result.
\begin{theorem}\label{Theorem:BijectionBivariables}\item
\begin{enumerate}
\item\label{ThmBiV1}
Every bivariable $\omega\in \kk[a,b][x,y]$ trivialises an $\A^2$-bundle $\rho_f\colon X_f\to \A^2_*$ for some $f\in \kk[a^{\pm 1},b^{\pm 1}][x]$. 

More precisely, given a bivariable $\omega\in \kk[a,b][x,y]$, there exist elements $\tau_a$, $\tau_b$ and $f(x)$ of $\kk[a^{\pm 1},b][x,y]$, $\kk[a,b^{\pm 1}][x,y]$ and $\kk[a^{\pm 1},b^{\pm 1}][x]$, respectively, such that 
$\kk[a^{\pm1},b][\omega,\tau_a]=\kk[a^{\pm1},b]$, $\kk[a,b^{\pm1}][\omega,\tau_b]=\kk[a,b^{\pm1}][x,y]$ and $\tau_a=\tau_b+f(\omega).$ Moreover, the variety $X_f$ trivialised by $\omega$ is uniquely defined up to isomorphism of $\A^1$-bundles.
\item\label{ThmBiV2}The map that associates to a bivariable $\omega$ a variety $X_f$ trivialised by $\omega$ induces a bijection 
\[\begin{array}{ccc}
\left\{\begin{array}{c}\text{bivariables up} \\ \text{to the action of } \\ \Aut_{\k[a,b]}(\kk[a,b][x,y])\end{array}\right\} &\stackrel{\kappa}{\longrightarrow}& \left\{\begin{array}{c}\text{Varieties $X_f$ which are } \\ \text{ trivial $\A^2$-bundles}\\ \text{up to isomorphisms} \\ \text{of $\A^1$-bundles}\end{array}\right\}\\
\omega & \mapsto & X_f\end{array}\]
\end{enumerate}
\end{theorem}
According to the above bijection, the varieties $X_f$ that define trivial $\A^1$-bundles correspond to the set of $\k[a,b]$-variables of $\k[a,b,x,y]$, i.e.~to the set of trivial bivariables (see Example~\ref{exple-trivialXf}). The polynomials of the form $a^mx+b^ny$ with $m,n\ge 1$ are easy examples of non-trivial bivariables. They  correspond  to the varieties $X_f$ with $f=\frac{x}{a^mb^n}$ (see Example~\ref{exple-bivariable1}).

In Section~\ref{section:construction_of_bivariables}, we describe a procedure to construct more bivariables starting from a given bivariable $\omega$. Indeed, if  $\tau_a$, $\tau_b$ and $f(x)$ are given as in the first assertion  in Theorem~\ref{Theorem:BijectionBivariables} and if  $m,n\ge 1$ are positive integers such that $a^mb^nf(x)\in \k[a,b,x,y]$, then we can define,  for all polynomials $Q\in \k[a,b][x]$, new bivariables $\tilde{w}$ and $\hat\omega$ by  
$\hat\omega=\omega+aQ(a^m\tau_a)$ and $\tilde{\omega}=\omega+bQ(b^n\tau_b)$  (see Proposition~\ref{Prop:NewBiV}).

Applying this procedure to the bivariable $ax+by^2$, we obtain new bivariables of the form $ax+b^2y+bP(x)$, $P(x)\in\kk[x]$, which correspond  to the $\A^2$-bundles over $\A^2_*$ associated with the fibrations $\pi_{P,n}$ when $n>\deg(P)$ (see Example~\ref{Ex:m=1}). Therefore, this  gives a simple proof  that the polynomials $v_{P,n}$
are $\k[x]$-variables for each $n>\deg(P)$, and in particular that the $n$-th V\'en\'ereau polynomials are $\k[x]$-variables for all $n\geq3$ (see Remark \ref{rem:Venereau3}).

In the case where $\deg(P)=2$ and $\mathrm{char}(\k)\neq2$, we can go one step further: Applying the procedure to $\omega=ax+b^2y+bP(x)$, we now obtain the new bivariable $\omega+\frac{a^5}{2c}(\frac{y}{a}+\frac{P(x)}{ab}-\frac{1}{ab}P(\frac{\omega}{a}))$ which corresponds to the  $\A^2$-bundle   associated with the fibration $\pi_{P,2}$ (see Lemma~\ref{Lemm:V2}). This proves that the polynomials $v_{P,2}$ are $\k[x]$-variables when $\deg(P)=2$ and  $\mathrm{char}(\k)\neq2$. In particular, this generalises and gives a different proof for Lewis's result stating that the second V\'en\'ereau polynomial $v_{z^2,2}$ is a  $\k[x]$-variable.

Although the bundles $\rho_{f_3}$ and $\rho_{f_2}$ were easily trivialised with this technique, we were unfortunately not able to go further and couldn't trivialise the bundle $\rho_{f_1}$ associated with the first V\'en\'ereau polynomial. Nonetheless, we can simplify it and show that it is equivalent to a bundle with the transition function of degree $3$ in $x$ (see Example~\ref{Example:simplify_V1degree3}) and also to another which is still of degree $4$ but has only three summands (see Example~\ref{Example:simplify_V1degree4}). Surprisingly,  we are able to trivialise the bundle $\rho_g$, where the function 
 \[g=\frac{x}{ab^2}-\frac{x^2}{a^3b}-\frac{x^3}{a^2b^2}-\frac{5}{4}\cdot\frac{x^4}{ab^3}\]
 differs from $f_1$ only by the coefficient $\frac{5}{4}$ in its last summand (see Example~\ref{Example:almost_V1}).

In Section~\ref{Sec:BlowUp}, we strengthen the result that the morphisms $\pi_{P,n}\colon \A^4\to \A^2$ are $\A^2$-bundles over $\A^2_*$, by proving the stronger fact that $\pi_{P,n}$ yields a locally trivial $\A^2$-bundle $\hat{\pi}_{P,n}\colon \hat{\A}^4\to \hat{\A}^2$, where  $\hat{\A}^2$ and $\hat{\A}^4$ are respectively obtained from $\A^2$ and $\A^4$ by blowing-up $(0,0)$ and the surface $\pi_{P,n}^{-1}((0,0))$ given by $x=y=0$ (see Theorem~\ref{Thm:BlowUp}). Proving that  $\pi_{P,n}$ is a locally trivial $\A^2$-bundle is then equivalent to prove that $\hat{\pi}_{P,n}$ is trivial, or to prove that $\hat{\pi}_{P,n}$ is trivial on a neighbourhood of the exceptional curve of $\hat{\A}^2$.

In Section~\ref{Sec:Xfmn1}, we study the family of varieties $X_f$ and their (non necessarily trivial) associated bundles $\rho_f\colon X_f\to \A^2_*$. In the case where the denominator of $f$ has degree at most $1$ in either $a$ or $b$, we give a direct criterion to decide whether $\rho_f$ is a trivial $\A^2$-bundle or not. 

\section{Local triviality over the punctured plane}\label{sec:localtrivialityofpiPn}

The aim of this section is to prove Theorem~\ref{Theorem:fcs_de_transition}. We shall indeed show that the restriction of every $\A^2$-fibration $\pi_{P,n}$ over the punctured plane $\A^2_*$ is a locally trivial $\A^2$-bundle which is moreover isomorphic to a bundle of the form $\rho_f\colon X_f\to \A^2_*$ for some explicit $f$ depending on $P$ and $n$. We will also prove that each polynomial $v_{P,n}$ is a so-called stable variable. \\

The following lemma already tells us that $v_{P,n}$ is a $\kk[x,x^{-1}]$-variable of the ring $\kk[x,x^{-1}][y,z,u]$, or equivalently that the restriction of $\pi_{P,n}$ to $U_a$ is a trivial $\A^2$-bundle.

\begin{lemma}\label{Lemma:triviality over Ua}
The rational map $\varphi_{P,n}\colon\A^4\to\A^4$ defined by 
\begin{align*}\varphi_{P,n}(x,y,z,u)=\Big(x, v_{P,n}&, xz+y(uy+P(z)),\\
&\qquad \dfrac{u}{x}-\frac{P(z+\dfrac{y(yu+P(z))}{x})-P(z)}{xy}\Big)\end{align*}
has Jacobian determinant $1$ and restricts to an automorphism of the complement $\A^4\setminus\{x=0\}$ of the hyperplane defined by the equation $x=0$. 
\end{lemma}

\begin{proof}One first checks by a straightforward computation that $\varphi_{P,n}$ is equal to the composition
\[\varphi_{P,n}=\varphi_4\circ\varphi_3\circ\varphi_2\circ\varphi_1\]
 of the following four birational transformations of $\A^4$:
\[\begin{array}{rcl}
\varphi_4\colon (x,y,z,u)&\mapsto &(x,y+x^nz,z,u),\\
\varphi_3\colon(x,y,z,u)& \mapsto & (x,y,z,y^{-1}(u-x^{-1}P(x^{-1}z))),\\
\varphi_2\colon(x,y,z,u)& \mapsto & (x,y,xz+yu,x^{-1}u),\\
\varphi_1\colon(x,y,z,u)& \mapsto & (x,y,z,yu+P(z)).
\end{array}\]
Since the Jacobian determinants of $\varphi_1,\varphi_2,\varphi_3,\varphi_4$ are equal to $1$, $1/y$, $1$, $y$, respectively and since all these maps fix $y$, it follows that the Jacobian determinant of $\varphi_{P,n}$ is equal to $1$. Moreover, since all components of $\varphi_{P,n}$ belong to $\kk[x^{\pm1},y,z,u]$ (remark that the numerator of last component's last summand is indeed divisible by $y$), all components of its inverse belong to $\kk[x^{\pm1},y,z,u]$ as well.
\end{proof}

As an immediate consequence, we get the following statement.

\begin{corollary}
Every map $\pi_{P,n}\colon\A^4\to\A^2$ is an $\A^2$-fibration and restricts to a trivial $\A^2$-bundle over the open set $(\A^1\setminus \{0\})\times \A^1$.
\end{corollary}

Another worth mentioning consequence of Lemma \ref{Lemma:triviality over Ua} is the fact that $v_{P,n}$ is a $1$-stable variable. This can be shown by a general argument due to El Kahoui and Ouali \cite{KO} (see also \cite{Fre}). 

\begin{proposition}\label{Prop:variables_stables}Every polynomial $v_{P,n}$ is a $\kk[x]$-variable of $\kk[x,y,z,u,t]$, where $t$ denotes a new indeterminate.
\end{proposition}

\begin{proof}We shall construct an automorphism $\Psi\in \Aut_{\kk[x]}(\kk[x,y,z,u,t])$ which maps $v_{P,n}$ onto \[w_s=v_{P,n}+x^st=y+x^st+x^{n}(xz+y^2u+yP(z)),\]
where $s$ denotes a suitable positive integer. The proposition will follow, since it is easy to check (see for example Lemma~\ref{lemme:variables_de_longueur2} below), that $w_s$ is a $\kk[x,z,u]$-variable of the ring $\kk[x,z,u][y,t]$.

The construction of $\Psi$ involves the rational map $\varphi=\varphi_{P,n}$ defined in Lemma~\ref{Lemma:triviality over Ua}. Since $\varphi$   restricts to an automorphism  outside the hyperplane $x=0$, its comorphism  $\varphi^*\colon\kk[x^{\pm1},y,z,u]\to\kk[x^{\pm1},y,z,u]$, $Q\mapsto Q\circ\varphi$, is a $\kk[x^{\pm1}]$-automorphism of the ring $\kk[x^{\pm1},y,z,u]$.  We denote by $F$ the extension of $\varphi^*$ as a $\kk[x^{\pm 1},t]$-automorphism of $\kk[x^{\pm 1},y,z,u,t]$. Note  that $F(y)=v_{P,n}$.

For each $\xi \in \k[x^{\pm 1},t]$, we  denote by $H_\xi$ the  $\kk[x^{\pm 1},z,u,t]$-automorphism of $\kk[x^{\pm 1},y,z,u,t]$  defined by $H_\xi(y)=y+\xi$, and by $\Phi_\xi$ the $\kk[x^{\pm 1},t]$-automorphism of $\kk[x^{\pm 1},y,z,u,t]$  defined by  $\Phi_\xi=F\circ H_\xi\circ F^{-1}$. Note that $\Phi_{\xi}(v_{P,n})=v_{P,n}+\xi$ by construction.

Observe that all elements $\Phi_\xi(y), \Phi_\xi(z), \Phi_\xi(u), \Phi_\xi(t)\in\kk[x^{\pm 1},y,z,u,t]$ depend polynomially on $\xi$. Moreover, since $\Phi_0$ is the identity, their coefficients of degree $0$ in $\xi$ are $y,z,u,t$, respectively. Therefore, choosing $\xi=x^s\cdot t$ for a large enough integer $s$, we obtain that $\Phi_{x^s\cdot t}(\k[y,z,u,t])\subseteq \k[x,y,z,u,t]$. Since $\Phi_{x^s\cdot t}$ is of Jacobian determinant  $1$, this implies  that $\Phi_{x^s\cdot t}$ restricts to a $\k[x]$-automorphism of $\k[x,y,z,u,t]$. Recall that $\Phi_{x^s\cdot t}(v_{P,n})=w_s$ by construction. This concludes the proof.
\end{proof}

We now proceed with the proof of Theorem~\ref{Theorem:fcs_de_transition}.

\begin{proof}[Proof of Theorem~$\ref{Theorem:fcs_de_transition}$]
Let $m$ be an integer such that $mn>\deg(P)$ and denote by $F_{P,n,m}$ the birational map of $\A^4$ defined by
\begin{align*}F_{P,n,m}(x,y,z,u)&=\left(x,y,z,u+\frac{z}{xy^2}-\frac{1}{xy^m}\cdot \frac{y^m-(x^nz)^m}{y-x^nz}P(\frac{z}{x})\right)\\&=(x,y,z,u+f_{P,n,m}(x,y,z)).\end{align*}
What we actually only need to prove is that the composition $F_{P,n,m}^{-1}\circ \varphi_{P,n}$, where $\varphi_{P,n}$ is given as in Lemma~$\ref{Lemma:triviality over Ua}$, restricts to an isomorphism between $\A^4\setminus \{v_{P,n}=0\}$ and $\A^4\setminus \{y=0\}$. 
For simplicity, let us denote $v=v_{P,n}$, $\varphi=\varphi_{P,n}$ and $F=F_{P,n,m}$.

Note that 
\[F^{-1}(x,y,z,u)=\left(x,y,z,u-\frac{z}{xy^2}+\frac{1}{xy^m}\cdot \frac{y^m-(x^nz)^m}{y-x^nz}P(\frac{z}{x})\right).\]
In particular, the first two components of $T^{-1}\circ \varphi$ and $T^{-1}$ are equal to $x$ and $v$, respectively. Let us denote their common third component by $$\omega=\frac{v-y}{x^n}=xz+y(uy+P(z)).$$ In order to prove the proposition, we only need to show that the last component of $F^{-1}\circ \varphi$ is an element of $\kk[x,y,z,u,v^{-1}]$. Indeed, since $F^{-1}\circ \varphi$ is a birational map of $\A^4$ whose second component is equal to $v$, this will imply that $T^{-1}\circ \varphi$ induces an isomorphism from $\A^4\setminus \{v=0\}$ to $\A^4\setminus \{y=0\}$.

Remark that the last component of $F^{-1}$ is an element of $\kk[x^{\pm1},y^{\pm1},z,u]$. Therefore, since the last component of 
\[\begin{array}{rcl}
\varphi(x,y,z,u)=\left(x,v, \omega ,\dfrac{1}{x}u+\dfrac{P(z)-P(\omega/x)}{xy}\right)
\end{array}\]
 is in fact an element of $\kk[x^{\pm 1},y,u,z]$, we only need to show that $x$ does not appear in the denominator of the last component of $T^{-1}\circ \varphi$.

This component is equal to
\begin{align*}
(F^{-1}\circ \varphi)^*(u)&=\varphi^{*}\left(u-\frac{z}{xy^2}+\frac{1}{xy^m}\cdot \frac{y^m-(x^nz)^m}{y-x^nz}P(\frac{z}{x})\right)\\
&=\frac{1}{x}u+\frac{P(z)-P(\omega/x)}{xy}-\frac{\omega}{xv^2}+\frac{1}{xv^m}\frac{v^m-(x^n\omega)^m}{y}P(\omega/x)\\
&=\frac{1}{x}u+\frac{P(z)}{xy}-\frac{\omega}{xv^2}-\frac{1}{xv^m}\frac{(x^n\omega)^m}{y}P(\omega/x)\\
&=\frac{1}{xyv^2}\cdot\left(uv^2y+v^2P(z)-\omega y\right)-\frac{\omega^m}{yv^m}\cdot x^{mn-1}P(\omega/x).
\end{align*}
Since we took $m$ such that $mn>\deg(P)$, we only need to check that
$$uv^2y+v^2P(z)-\omega y\equiv0\pmod{x}.$$
For this, we use the fact that $v\equiv y$ and $\omega\equiv y(uy+P(z))$ modulo $(x)$ and find
\[uv^2y+v^2P(z)-\omega y\equiv uy^3+y^2P(z)-y^2(uy+P(z))\equiv 0\pmod{x}.\]
\end{proof}

\begin{example}\label{Exa:FforVenereau}
Recall that the $n$-th V\'en\'ereau polynomial is defined by
\[V_n=v_{z^2,n}=y+x^n(xz+y^2u+yz^2).\] By Theorem~\ref{Theorem:fcs_de_transition} above, the map $\A^4\setminus\A^2\to\A^2_*$, $(x,y,z,u)\mapsto(x,V_n)$ is an $\A^2$-bundle isomorphic to $\rho_{f_n}\colon X_{f_n}\to \A^2_*$, where
\[f_n=\frac{x}{ab^2}-\frac{1}{ab^m}\cdot \frac{b^m-(a^nx)^m}{b-a^nx}\left(\frac{x}{a}\right)^2,\] with $m$ such that $mn>2$.
If $n\geq3$, we can choose $m=1$ and we get 
 \[f_n=\frac{x}{ab^2}-\frac{1}{ab^1}\cdot \frac{b^1-(a^nx)^1}{b-a^nx}\cdot\frac{x^2}{a^2}=\dfrac{x}{ab^2}-\dfrac{x^2}{a^3b} \quad \text{ when }n\ge 3.\]
 For $n=2$, we can choose $m=2$ and get
\begin{align*}f_2=\frac{x}{ab^2}-\frac{1}{ab^2}\cdot \frac{b^2-(a^2x)^2}{b-a^2x}\cdot\frac{x^2}{a^2}&=\frac{x}{ab^2}-\frac{1}{ab^2}\cdot (b+a^2x)\cdot\frac{x^2}{a^2}\\&= \dfrac{x}{ab^2}-\dfrac{x^2}{a^3b}-\dfrac{x^3}{ab^2}.\end{align*}
For $n=1$, we choose $m=3$ and get
\begin{align*} f_1=\frac{x}{ab^2}-\frac{1}{ab^3}\cdot \frac{b^3-(ax)^3}{b-ax}\cdot\frac{x^2}{a^2}&=\frac{x}{ab^2}-\frac{1}{ab^3}\cdot (b^2+axb+a^2x^2)\cdot\frac{x^2}{a^2}\\&=\dfrac{x}{ab^2}-\dfrac{x^2}{a^3b}-\dfrac{x^3}{a^2b^2}-\dfrac{x^4}{ab^3}.\end{align*}
These functions are exactly those computed by Kaliman and Zaidenberg in \cite[Proposition 2]{KZ2004}.
\end{example}

\begin{example}\label{Ex:f(x)_for_m=1,2}We now consider   the simple cases where $n>\deg(P)$ and $n=\deg(P)$.
\begin{enumerate}
\item If $n>\deg(P)$, then  we can choose $m=1$ in Theorem~\ref{Theorem:fcs_de_transition}. Hence, the restriction of $\pi_{P,n}$ to $\A^2_*$ is isomorphic to $\rho_f\colon X_f\to \A^2_*$ with
\[f(x)=\frac{x}{ab^2}-\frac{1}{ab^1}\cdot \frac{b^1-(a^nx)^1}{b-a^nx}P(\frac{x}{a})=\frac{x}{ab^2}-\frac{1}{ab}P(\frac{x}{a}).\] 
\item If $n=\deg(P)$, then we can choose $m=2$ and the map $(x,y,z,u)\mapsto(x,v_{P,n})$ has the structure of an $\A^2$-bundle over $\A^2_*$ isomorphic to $\rho_{f}\colon X_f\to \A^2_*$ with
\begin{align*}f(x)&=\frac{x}{ab^2}-\frac{1}{ab^2}\cdot \frac{b^2-(a^nx)^2}{b-a^nx}P(\frac{x}{a})\\
&=\frac{x}{ab^2}-\frac{1}{ab^2}\cdot (b+a^nx)P(\frac{x}{a})\\
&=\frac{x}{ab^2}-\frac{1}{ab}P(\frac{x}{a})-\frac{a^{n-1}}{b^2}xP(\frac{x}{a}).
\end{align*} 
\end{enumerate}
\end{example}

\section{Bivariables and their relationship with trivial \texorpdfstring{$\A^1$}{A1}-bundles -- the proof of Theorem~\ref{Theorem:BijectionBivariables}}\label{Sec:Bivariables}

We start by giving an easy result which was already noticed   in \cite{KZ2004}. 
\begin{lemma}\label{Lem:Equivalenttrivialities}
Let $n\ge 1$ and let $P\in \k[z]$. Then, the following statements are equivalent:
\begin{enumerate}
\item\label{visxvariable} The polynomial $v_{P,n}$ is a $\k[x]$-variable of $\kk[x,y,z,u]$.
\item\label{piPntrivial} The morphism $\pi_{P,n}\colon\A^4\to\A^2$ is a trivial $\A^2$-bundle.
\item\label{piPntrivialorigin} There exists a neighbourhood of the origin in $\A^2$ over which $\pi_{P,n}$ is a trivial $\A^2$-bundle.
\item\label{piPntrivialoutsideorigin} The restriction of the $\A^2$-fibration $\pi_{P,n}$ over the punctured plane $\A^2_*$ is a trivial $\A^2$-bundle.
\end{enumerate}
\end{lemma}
\begin{proof}
Asking that $\pi_{P,n}\colon\A^4\to\A^2$ is a trivial $\A^2$-bundle means exactly that there exist $r,s\in\k[x,y,z,u]$ such that the map
 \[\A^4\to \A^2\times \A^2=\A^4, (x,y,z,u)\mapsto(x,v_{P,n},r(x,y,z,u),s(x,y,z,u))\]
is an isomorphism, i.e.~ such that the equality $\k[x,y,z,u]=[x,v_{P,n},r,s]$ holds. Hence, the statements  \ref{visxvariable} and \ref{piPntrivial} of the lemma are equivalent.

If the morphism $\pi_{P,n}\colon\A^4\to\A^2$ is a trivial $\A^2$-bundle, then its restriction to every subset of $\A^2$ is trivial. Therefore,   Assertion \ref{piPntrivial} implies both  \ref{piPntrivialorigin}  and  \ref{piPntrivialoutsideorigin}. It remains to prove the converse implications.

By Theorem~\ref{Theorem:fcs_de_transition}, the map $\pi_{P,n}$ defines a locally trivial $\A^2$-bundle over the punctured plane.  Hence, if assertion \ref{piPntrivialorigin} is true, then $\pi_{P,n}$ is a locally trivial $\A^2$-bundle over the whole $\A^2$, which is  isomorphic to a vector bundle by Bass-Connell-Wright Theorem \cite{BCW} and is furthermore trivial by Quillen-Suslin Theorem \cite{Qui,Sus}. This shows that Assertion  \ref{piPntrivialorigin} implies \ref{piPntrivial}.

Under Assertion $\ref{piPntrivialoutsideorigin}$, there exist $r,s\in\k(x,y,z,u)$  such that the  map
 \[\A^4\setminus \{(x=v_{P,n}=0)\} \to \A^2_*\times \A^2, (x,y,z,u)\mapsto(x,v_{P,n},r(x,y,z,u),s(x,y,z,u))\]
 is an isomorphism. Since the locus where $x=v_{P,n}=0$  has codimension $2$ in $\A^4$ (as it is in fact the surface defined by the  equations $x=y=0$ in $\A^4$), the ring of regular functions on $\A^4\setminus \{(x=v_{P,n}=0)\}=\A^2_*\times \A^2$ is equal to the whole ring $\k[x,y,z,u]$. Hence, the map above induces an automorphism of $\k[x,y,z,u]$. In particular, we have that  $r,s\in \k[x,y,z,u]$ and $\k[x,y,z,u]=\k[x,v_{P,n},r,s]$. This proves that $\ref{piPntrivialoutsideorigin}$ implies $\ref{piPntrivial}$.
\end{proof}

By Theorem~\ref{Theorem:fcs_de_transition}, the four conditions of Lemma~\ref{Lem:Equivalenttrivialities} are equivalent to the fact that the $\A^2$-bundle $\rho_f\colon X_f\to\A^2_*$, where $f\in\kk[a^{\pm1},b^{\pm1}][x]$ is explicitly given in the statement of Theorem~\ref{Theorem:fcs_de_transition}, is a trivial $\A^2$-bundle. 

 We recall the notation $U_a=\A^2\setminus\{a=0\}$, $U_b=\A^2\setminus\{b=0\}$ and $U_{ab}=U_a\cap U_b=\A^2\setminus \{ab=0\}$. Let us denote by $G_a$, $G_b$ and $G_{ab}$ the automorphism groups  of the $\A^2$-bundles  $U_a\times \A^2$, $U_b\times \A^2$ and $U_{ab}\times \A^2$, respectively. In particular, remark that $G_a$ and $G_b$ are subgroups of $G_{ab}$. In the sequel, we will abuse notation and denote an element \[((a,b),(x,y))\mapsto ((a,b),(F(a,b,x,y),G(a,b,x,y)))\] 
by $(F(a,b,x,y),G(a,b,x,y))$ or simply by $(F(x,y),G(x,y))$.  

\begin{lemma}\label{Lem:isomorphic-A2bundles}Let $f(x),g(x)\in\kk[a^{\pm 1},b^{\pm 1}][x]$. Then, the two $\A^2$-bundles $\rho_f$ and $\rho_{g}$ are isomorphic as $\A^2$-bundles if and only if there exist $\alpha\in G_a$ and $\beta\in G_b$ such that 
\[\alpha\circ (x,y+f(x))\circ \beta^{-1}=(x,y+g(x)).\]
In particular, the $\A^2$-bundle $\rho_f$ is isomorphic to the trivial bundle if and only if there exist $\alpha\in G_a$ and $\beta\in G_b$ such that \[\alpha\circ \beta^{-1}=(x,y+f(x)).\]
\end{lemma}
\begin{proof}
Recall that the $\A^2$-bundles $\rho_f$ and $\rho_g$ are constructed by gluing $U_a\times \A^2$ and $U_b\times \A^2$  via the transition functions  $(x,y+f(x))\in G_{ab}$ and $(x,y+g(x))\in G_{ab}$, respectively. Hence, they are isomorphic as $\A^2$-bundles if and only if one can find automorphisms of the $\A^2$-bundles  $U_a\times \A^2$ and $U_b\times \A^2$ that are compatible with the gluing. This gives the result.\end{proof}

Investigating on the conditions of Lemma~\ref{Lem:isomorphic-A2bundles} for a specific example (for instance in the case corresponding to the first V\'enereau polynomial) is not a simple task. On the contrary, when we consider $X_f$ and $X_g$ as $\A^1$-bundles over $\A^2_{*}\times \A^1$ (via the projection onto the first three coordinates), it becomes easy to decide whether two bundles $X_f$ and $X_g$ are equivalent as $\A^1$-bundles.   


\begin{lemma}\label{Lemm:A1bundle}
For all $f,g\in \kk[a^{\pm 1},b^{\pm 1}][x]$, the following conditions are equivalent.
\begin{enumerate}
\item\label{A1bundle1}
The varieties $X_f$ and $X_{g}$ are isomorphic as $\A^1$-bundles.
\item\label{A1bundle2}
There exist $\tau_a\in \kk[a^{\pm 1},b][x,y]$ and $\tau_b\in \kk[a,b^{\pm 1}][x,y]$ such that $\alpha=(x,\tau_a)\in G_a$, $\beta=(x,\tau_b)\in G_b$ and \[\alpha\circ (x,y+f(x))\circ \beta^{-1}=(x,y+g(x)).\]
\item\label{A1bundle3}
There exist $r_a\in \kk[a^{\pm 1},b][x]$, $r_b\in \kk[a,b^{\pm 1}][x]$ and $\lambda \in \kk^*$ such that \[g(x)=\lambda f(x)+r_a(x)+r_b(x).\]
\end{enumerate}
\end{lemma}
\begin{proof}The $\A^1$-bundle structures $X_f\to \A^2_{*}\times \A^1$ and $X_g\to \A^2_{*}\times \A^1$ are given by the restriction of the projection $\A^2_{*}\times \A^2\to \A^2_{*}\times \A^1, (a,b,x,y)\mapsto(a,b,x)$
on both charts $U_a\times \A^2$ and $U_b\times \A^2$. An isomorphism of $\A^1$-bundles between $X_f$ and $X_g$ is then given by $\alpha\in \Aut(U_a\times \A^2)$ and $\beta\in \Aut(U_b\times \A^2)$, both compatible with that projection. These automorphisms $\alpha$ and $\beta$ must also belong to $G_a$ and $G_b$, respectively. Moreover, their   first coordinate must be equal to $x$. This shows that Assertions \ref{A1bundle1} and \ref{A1bundle2} are equivalent.

Suppose now that  $\alpha,\beta$ are as in \ref{A1bundle2}. Since the Jacobian determinants of $\alpha$ and $\beta$ do not vanish on $U_a$ and $U_b$, respectively, they are of the form 
\[\alpha=(x,\mu_a y+r_a(x)) \text{ and } \ \beta^{-1}=(x,\mu_b y+r_b(x))\]
for some $\mu_a \in \kk[a^{\pm 1}]^*$, $\mu_b\in \kk[b^{\pm 1}]^*$, $r_a\in \kk[a^{\pm 1},b][x]$ and $r_b\in \kk[a,b^{\pm 1}][x]$. The equality $\alpha \circ (x,y+f(x))\circ \beta^{-1}=(x,y+g(x))$ is then equivalent to 
\[\mu_a\mu_b=1, r_a(x)+\mu_a r_b(x)+ \mu_a f(x) =g(x).\]
The above equalities  can only occur when $\mu_a=(\mu_b)^{-1}\in \kk^*$. Therefore, Assertion  $\ref{A1bundle2}$ implies Assertion  \ref{A1bundle3}. Finally, it is easy to check that \ref{A1bundle3} implies \ref{A1bundle2}, as we can construct suitable $\alpha$ and $\beta^{-1}$ given $\lambda,r_a,r_b$ such that $g(x)=\lambda f(x)+r_a(x)+r_b(x)$.
\end{proof}

\begin{remark}\label{Rem:intro-bivariables} Suppose that  $\rho_f\colon X_f\to\A^2_*$ is isomorphic to the trivial $\A^2$-bundle. Then, by Lemma~\ref{Lem:isomorphic-A2bundles}, there exist $\alpha\in G_a$ and $\beta\in G_b$ such that $\alpha=(x,y+f(x))\circ\beta$. In particular,   $\alpha$ and $\beta $   have the same first component, which  is in fact an element of $\kk[a,b][x,y]$. We call such an element a bivariable.
\end{remark}

\begin{definition}\label{Defi:bivariable}
We say that an element $\omega\in \kk[a,b][x,y]$ is a \emph{bivariable} if it both a $\kk[a^{\pm1},b]$-variable of $\kk[a^{\pm1},b,x,y]$ and a $\kk[a,b^{\pm1}]$-variable of $\kk[a,b^{\pm1},x,y]$. 
\end{definition}

\begin{example}\label{exple-bivariable0.0}
Every  $\k[a,b]$-variable of $\k[a,b,x,y]$ is a bivariable.
\end{example}

\begin{example}\label{exple-bivariable1.0}Let $m,n\geq1$ be positive integers. Then, the polynomial $\omega=a^mx+b^ny$ is a bivariable. Indeed, choosing $\tau_a=a^{-m}y$ and $\tau_b=-b^{-n}x$, we can define $\alpha=(\omega,\tau_a)\in G_a$ and  $\beta=(\omega,\tau_b)\in G_b$. We remark that  $\alpha\circ\beta^{-1}=(x,y+f(x))$ with $f(x)=\frac{x}{a^mb^n}$.
\end{example}

As  explained above, we can associate a bivariable to every trivial bundle  $\rho_f\colon X_f\to\A^2_*$. This motivates the following definition.

\begin{definition}\label{Defi:associated-bivariable}
We say that a bivariable $\omega\in\kk[a,b][x,y]$ trivialises the bundle $\rho_f\colon X_f\to\A^2_*$, $f(x)\in \kk[a^{\pm 1},b^{\pm 1}][x]$, if there exist elements $\tau_a$ and $\tau_b$   in $\kk[a^{\pm 1},b][x,y]$ and $\kk[a,b^{\pm 1}][x,y]$, respectively,  such that $\kk[a^{\pm1},b][\omega,\tau_a]=\kk[a^{\pm1},b]$, $\kk[a,b^{\pm1}][\omega,\tau_b]=\kk[a,b^{\pm1}][x,y]$ and $\tau_a=\tau_b+f(\omega)$. 
\end{definition}

\begin{remark}\label{Rem:trivialazed_bundles_are_trivial}
If a bivariable $\omega\in\kk[a,b][x,y]$ trivialises a bundle $\rho_f\colon X_f\to\A^2_*$ and if $\tau_a$, $\tau_b$ and $f$ are as in Definition~\ref{Defi:associated-bivariable}, then    \[\alpha\circ\beta^{-1}=(x,y+f(x)),\]
where $\alpha=(\omega,\tau_a)\in G_a$ and $\beta=(\omega,\tau_b)\in G_b$.
In particular, $\rho_f$ is isomorphic to the trivial $\A^2$-bundle.
\end{remark}

\begin{lemma}\label{lemme-def-tau_a-tau_b}
Every bivariable $\omega\in \kk[a,b][x,y]$ trivialises a bundle $\rho_f\colon X_f\to\A^2_*$ for some $f(x)\in \kk[a^{\pm 1},b^{\pm 1}][x]$.
\end{lemma}

\begin{proof}
By definition, an element $\omega\in \kk[a,b][x,y]$ is a bivariable if and only if there exist $\tau_a\in \kk[a^{\pm 1},b][x,y]$ and $\tau_b\in \kk[a,b^{\pm 1}][x,y]$ such that 
\[\kk[a^{\pm1},b][\omega,\tau_a]=\kk[a^{\pm1},b]\text{ and }\kk[a,b^{\pm1}][\omega,\tau_b]=\kk[a,b^{\pm1}][x,y],\]
i.e.~such that $\alpha=(\omega,\tau_a)\in G_a$ and $\beta=(\omega,\tau_b)\in G_b$.  The Jacobian determinant of $\alpha$ is an element of $\kk[a^{\pm1},b]$ which does not vanish on $U_a \times \A^2$. Hence, $\Jac(\alpha)\in\kk[a^{\pm 1}]\setminus \{0\}$. Substituting $\tau_a$ with $\frac{\tau_a}{\Jac(\alpha)}$, we may thus assume that $\Jac(\alpha)=1$. Similarly, we may assume that $\Jac(\beta)=1$. Moreover, since the first components of $\alpha$ and  $\beta$ are both equal to $\omega$, the first component of the composition $\alpha\circ\beta^{-1}\in G_{ab}$ is equal to $x$. As the Jacobian determinant of $\alpha\circ\beta^{-1}$ is equal to $1$, it follows that $\alpha\circ\beta^{-1}=(x,y+f(x))$ for some $f\in \kk[a^{\pm 1},b^{\pm 1}][x]$. 

Finally,   the equality $\alpha=(x,y+f(x))\circ \beta$ implies that $\tau_a=\tau_b+f(\omega)$ as desired.
\end{proof}

The group   of $\kk[a,b]$-automorphisms of the ring $\kk[a,b][x,y]$ naturally acts on the set of bivariables. Indeed, if $\omega\in \kk[a,b][x,y]$ is a bivariable and if  $g\in \Aut_{\k[a,b]}(\kk[a,b][x,y])$ is an automorphism, then the element $g(\omega)$ is again a bivariable. Therefore, we may consider bivariables up to the action of the group $\Aut_{\k[a,b]}(\kk[a,b][x,y])$ and introduce the following definition.

\begin{definition}\label{Rem:ActionAutOnBivariables}We say that two bivariables $\omega_1, \omega_2\in\kk[a,b][x,y]$ are \emph{equivalent} if there exists a $\kk[a,b]$-automorphism of the ring $\kk[a,b][x,y]$ that maps $\omega_1$ onto $\omega_2$.
\end{definition} 

We now proceed with the proof of Theorem~\ref{Theorem:BijectionBivariables}.


\begin{proof}[Proof of Theorem~$\ref{Theorem:BijectionBivariables}$]
\ref{ThmBiV1}: By Lemma~\ref{lemme-def-tau_a-tau_b} and Remark~\ref{Rem:trivialazed_bundles_are_trivial}, every bivariable trivialises an $\A^2$-bundle $\rho_f\colon X_f\to\A^2_*$ isomorphic to the trivial bundle. 

We now prove that the isomorphism class of $X_f$, as a $\A^1$-bundle, is uniquely determined by $\omega$. Suppose that a bivariable $w$ trivialises two such bundles $\rho_f$ and $\rho_{\tilde{f}}$ and consider two suitable pairs $(\tau_a,\tau_b)$ and $(\tilde{\tau_a},\tilde{\tau_b})$ as in Definition~\ref{Defi:associated-bivariable}. Define $\alpha=(\omega,\tau_a)$, $\tilde{\alpha}=(\omega,\tilde{\tau_a})\in G_a$,  $\beta=(\omega,\tau_b)$ and $\tilde{\beta}=(\omega,\tilde{\tau_b})\in G_b$. Then, $\alpha\circ\beta^{-1}=(x,y+f(x))$ and $\tilde{\alpha}\circ\tilde{\beta}^{-1}=(x,y+\tilde{f}(x))$. Since 
 \[\alpha\circ\tilde{\alpha}^{-1}=(x,y+r_a(x)) \text{ and } \beta\circ\tilde{\beta}^{-1}=(x,y+r_b(x))\]
 for some $r_a\in \kk[a^{\pm 1},b][x]$ and $r_b\in \kk[a,b^{\pm 1}][x]$, the two varieties $X_f$ and $X_{\tilde{f}}$ are isomorphic as $\A^1$-bundles by Lemma~\ref{Lemm:A1bundle}.
 
\ref{ThmBiV2}: Suppose that two bivariables $\omega'$ and $\omega$ are equivalent and let  $\varphi\in \Aut_{\kk[a,b]}(\kk[a,b][x,y])$ be such that $\varphi(\omega)=\omega'$. Then, $\psi=(\varphi(x),\varphi(y))\in G_a\cap G_b$. Choosing $\alpha=(\omega,\tau_a)\in G_a$ and $\beta=(\omega,\tau_b)\in G_b$ as above, we obtain 
 \[\alpha'=\alpha\circ \psi=(\omega',\tau_a')\in G_a \text{ and } \beta'=\beta\circ \psi=(\omega',\tau_b')\in G_b\]
 for suitable elements $\tau_a'$, $\tau_b'$.
 Since, $\alpha\circ\beta^{-1}=\alpha'\circ\beta'^{-1}$, this implies that $\omega$ and $\omega'$ trivialise isomorphic (as $\A^1$-bundles) $\A^2$-bundles. Hence, the map $\kappa$ of Assertion~\ref{ThmBiV2} is well defined. It remains to prove that this map is bijective. 
 
The surjectivity of $\kappa$ follows from Remark~\ref{Rem:intro-bivariables}. Finally, we prove the injectivity. Consider two bivariables $\omega$ and $\omega'$ trivialising the same bundle $\rho_f$. Let $\alpha=(\omega,\tau_a)\in G_a$, $\beta=(\omega,\tau_b)\in G_b$, $\alpha'=(\omega',\tau_a')\in G_a$ and $\beta'=(\omega',\tau_b')\in G_b$ be such that $\alpha\circ\beta^{-1}=\alpha'\circ\beta'^{-1}=(x,y+f(x))$. Then, the element $\psi=\alpha^{-1}\circ\alpha'=\beta^{-1}\circ \beta'\in G_a\cap G_b$ is  an automorphism of $\A^2\times \A^2$, whose action on $\k[a,b,x,y]$ gives an automorphism $\psi^{*}\in\Aut_{\kk[a,b]}(\kk[a,b][x,y])$ sending $\omega$ onto $\omega'$. This shows that $\omega$ and $\omega'$ are equivalent and concludes the proof. 
\end{proof}

We finish this section by considering again two simple families of bivariables.
\begin{example}\label{exple-trivialXf}
Let $v\in \k[a,b,x,y]$ be a $\k[a,b]$-variable and let  $\tau\in \k[a,b,x,y]$ be such that $\k[a,b][x,y]=\k[a,b][v,\tau]$. Defining $\tau_a=\tau_b=\tau$, the bijection of Theorem~\ref{Theorem:BijectionBivariables} associates the (equivalence class) of the bivariable $v$ to the (isomorphism class as $\A^1$-bundle) of the trivial bundle $X_f$ with $f=0$.
\end{example}

\begin{example}\label{exple-bivariable1}Let $m,n\geq1$ be positive integers and $P\in \k[a,b]$. Then, the polynomial $\omega=a^mx+b^ny+P$ is a bivariable. Indeed, choosing $\tau_a=a^{-m}y$ and $\tau_b=-b^{-n}x$, we get that $\alpha=(\omega,\tau_a)\in G_a$ and $\beta=(\omega,\tau_b)\in G_b$. Then, the bijection of Theorem~\ref{Theorem:BijectionBivariables} sends (the class of) this bivariable onto (the class of) $X_f$ where $f=f(x)=\frac{x-P}{a^mb^n}$, since $\tau_a=\tau_b+f(\omega)$.
\end{example}

\section{A procedure to construct bivariables}\label{section:construction_of_bivariables}

 Let $\omega\in\kk[a,b][x,y]$ be a bivariable, and let $\tau_a\in\kk[a^{\pm1},b][x,y]$ and $\tau_b\in\kk[a,b^{\pm1}][x,y]$ be such that $\alpha=(\omega,\tau_a)\in G_a$ and $\beta=(\omega,\tau_b)\in G_b$. An easy way to get a new bivariable from $\omega$ is to compose $\alpha$ with a triangular automorphism $\varphi=(x+P(y),y)$, where $P\in\kk[a,b][x]$ should be well chosen, so that $\tilde{\omega}=\omega+P(\tau_a)$ is also a $G_b$-variable. (Note that $\tilde{\omega}$ is the first component of $\varphi\circ\alpha$, hence it is a $G_a$-variable.) In order to see that this idea works in general, we recall the following well-known result (see for example the proof of  \cite[Theorem 4]{EV}).
 
 \begin{lemma}\label{lemme:variables_de_longueur2}
 Let $R$ be a commutative ring with unity, $m\geq1$ be a positive integer, $a\in R$ be a non-zero divisor and $f,Q\in R[X]$ be polynomials in one indeterminate with coefficient in $R$. Then, the polynomial 
 \[v=x+aQ(a^my+f(x))\]
is a $R$-variable of $R[x,y]$.
 \end{lemma} \begin{proof}
Our proof simply follows the proof of \cite[Theorem 4]{EV}. We shall  construct a polynomial $g(x)\in R[x]$ such that
\begin{equation}\label{gvfx} f(x)\equiv_m g(v) \tag{$\spadesuit$} \end{equation}
and show, by giving its inverse, that the map $\varphi \in \mathrm{End}_R(R[x,y])$, defined by 
$\varphi(x)=v \text{ and } \varphi(y)=y+(f(x)-g(v))/a^m$, is an automorphism.

The idea is to work inductively on $i\in \{2,\ldots,m\}$. 
Let us denote by $\equiv_i$ the congruence in $R[x]$ modulo $a^i$ and let us define polynomials $g_1,\ldots,g_m\in R[x]$  by 
$g_1=f$ and $g_i=f(x-aQ(g_{i-1}(x)))$ for $i\in \{2,\ldots,m\}.$

We claim that \[g_i(v)\equiv_{i}  f(x)\]  for all $i=1,\ldots,m$. 
This is indeed true for $i=1$,  since    $f(v)=g_1(v)\equiv_1 g_1(x)=f(x)$, and then also for all $i\in \{2,\ldots,m\}$, since  
\[g_i(v)=f(v-aQ(g_{i-1}(v)))\equiv_i f(v-aQ(f(x)))\equiv_m f(x).\] 

In particular, the polynomial  $g(x)=g_m(x)\in R[x]$ satisfies the desired congruence \eqref{gvfx} and the map $\varphi$  defined above is an $R$-endomorphism of  $R[x,y]$.

We now proceed to construct the inverse map of $\varphi$. For this, we define $\hat{v}=x-aQ(a^my+g(x))$ and claim that $f(\hat{v})\equiv_m g(x)$. This follows from  the congruences $f(x-aQ(g(x)))\equiv_i g(x)$  that we prove now by induction on  $i=1,\ldots,m$. 
For $i=1$, this holds, since $f(x-aQ(g(x)))\equiv_1 f(x)\equiv_m g(v)\equiv_1 g(x)$. For $i>1$, we replace $x$ with $x-aQ(g(x))$ in \eqref{gvfx} and obtain 
\[\begin{array}{rcl}
f(x-aQ(g(x)))&\equiv_m &g(x-aQ(g(x))+aQ(f(x-aQ(g(x)))))\\
&\equiv_i&g(x-aQ(g(x))+aQ(g(x))))=g(x)
\end{array}\]

Finally, we define $\psi\in \mathrm{End}_R(R[x,y])$ by 
 $\psi(x)=\hat{v} \text{ and } \psi(y)=y+(g(x)-f(\hat{v}))/a^m$. It is straightforward to check that $\varphi\circ\psi=\psi\circ \varphi=\mathrm{id}_{R[x,y]}.$ This concludes the proof.
\end{proof}

We now apply the above lemma to construct new bivariables.

\begin{proposition}\label{Prop:NewBiV}
Let  $\omega\in \kk[a,b][x,y]$ be a bivariable. According to Theorem~\ref{Theorem:BijectionBivariables}, let $\tau_a$,  $\tau_b$ and $f(x)$ be  such that 
$\kk[a^{\pm1},b][\omega,\tau_a]=\kk[a^{\pm1},b][x,y]$, $\kk[a,b^{\pm1}][\omega,\tau_b]=\kk[a,b^{\pm1}][x,y]$ and $\tau_a=\tau_b+f(\omega)$.  Suppose that $m,n\ge 1$ are positive integers such that $a^mb^nf(x)\in \k[a,b,x]$ and let $Q\in \k[a,b][x]$ be a polynomial. Then, the elements
\[\omega+aQ(a^m\tau_a)=\omega+aQ(a^m\tau_b+a^mf(\omega))\]
and
\[\omega+bQ(b^n\tau_b)=\omega+bQ(b^n\tau_a-b^nf(\omega))\]
are bivariables of $\kk[a,b][x,y]$.
\end{proposition}
\begin{proof}The proof that $ \omega+bQ(b^n\tau_b)$ is a bivariable being similar (by exchanging the roles of $a$ and $b$), we only prove that \[\hat\omega=\omega+aQ(a^m\tau_a)=\omega+aQ(a^m\tau_b+a^mf(\omega))\] is a bivariable.

 On the one hand, it is a $\kk[a^{\pm1},b]$-variable since it is the first component of the composition of the  $\kk[a^{\pm1},b]$-automorphisms defined by $(x,y)\mapsto(x+aQ(a^my),y))$ and by $(x,y)\mapsto(\omega,\tau_a)$. 
 
 On the other hand, there exists, by Lemma~\ref{lemme:variables_de_longueur2}, a $\kk[a,b^{\pm1}]$-automorphism, say $\varphi$, of $\kk[a,b^{\pm1}][x,y]$ whose first component is equal to $v=x+aQ(a^my+a^mf(x))$. The element $\hat\omega=\omega+aQ(a^m\tau_b+a^mf(\omega))$ is therefore a  $\kk[a,b^{\pm1}]$-variable, since it is equal to the first component of the composition of $\varphi$ with the  $\kk[a,b^{\pm1}]$-automorphism defined by  $(x,y)\mapsto(\omega,\tau_b)$. 
\end{proof}
\begin{remark}
Let $\omega$ be a bivariable associated with the data $\tau_a$, $\tau_b$ and $f(x)$ as in Theorem~\ref{Theorem:BijectionBivariables}. Let $\hat\omega$ be a new bivariable obtained from $\omega$ by applying Proposition~\ref{Prop:NewBiV}. Then, one can  easily compute   elements  $\hat{\tau}_a,\hat{\tau}_b$ and $\hat{f}(x)$ associated with $\hat{\omega}$. Indeed, in the proof of Lemma~\ref{lemme:variables_de_longueur2}, we gave an explicit automorphism of $R[x,y]$ whose first component is equal to $v$. 
\end{remark}
 
\begin{example}\label{Ex:m=1}
Let us consider the bivariable  $\omega=ax+b^2y$ from  Example~\ref{exple-bivariable1.0}. Recall that it is indeed a bivariable, associated with $\tau_a=\frac{y}{a}$, $\tau_b=-\frac{x}{b^2}$ and $f(x)=\frac{x}{ab^2}$.  By Proposition~\ref{Prop:NewBiV}, the polynomial \[\hat{\omega}=\omega+bP(-b^2\tau_b)=ax+b^2y+bP(x)\] is a bivariable for each $P(x)\in\kk[x]$. 

Furthermore, one claims that this new bivariable  is associated with 
\[\hat{f}(x)=\frac{x}{ab^2}-\frac{P(\frac{x}{a})}{ab},\]
with $\hat{\tau}_b=\tau_b=-\frac{x}{b^2}$ and with \[\hat{\tau}_a=\hat{\tau}_b+\hat{f}(\hat{\omega})=\frac{y}{a}+\frac{P(x)}{ab}-\frac{1}{ab}P(x+\frac{b^2y+bP(x)}{a}).\]
Our claims is   easy to check. First, it is   straightforward to see that $\hat{\tau}_a\in \k[a^{\pm 1},b,x,y]$. Then, since $(\hat{\omega},\hat{\tau}_b)\in G_b\subset G_{ab}$ is an automorphism of Jacobian determinant $1$, it follows that the element $(\hat{\omega},\hat{\tau}_a)\in G_{ab}$, whose both components belong to $\k[a^{\pm 1},b,x,y]$, is also of Jacobian determinant $1$, hence that it is an element of $G_a$.
\end{example} 

\begin{remark}\label{rem:Venereau3}
We recall that, as computed in   Example~\ref{Ex:f(x)_for_m=1,2}, the function $\hat{f}(x)$ above is actually the transition function of the bundle $\rho_{\hat{f}}$ corresponding to the fibration $\pi_{P,n}$ in the case where $n>\deg(P)$. Consequently, it follows from Theorem \ref{Theorem:BijectionBivariables} and Lemma \ref{Lem:Equivalenttrivialities} that every polynomial $v_{P,n}$ is a $\k[x]$-variable, when  $n>\deg(P)$. In particular, we recover the fact that the $n$-th V\'en\'ereau polynomials are $\k[x]$-variables for all $n\geq3$.
\end{remark}

We can now start with the bivariable of Example~\ref{Ex:m=1} and apply again Proposition~\ref{Prop:NewBiV} to it. Doing so, we construct the following new bivariables. 

\begin{lemma}\label{Lemm:V2}
Let $P\in \k[X]$ be a polynomial of degree $2$ with leading coefficient $c\in \k^*$ and suppose that $\mathrm{char}(\k)\not=2$. Then, the polynomial 
\[\omega+\frac{a^5}{2c}(\frac{y}{a}+\frac{P(x)}{ab}-\frac{1}{ab}P(\frac{\omega}{a}))\]
is a bivariable,  associated with the $\A^2$-bundle over $\A^2_*$ corresponding to the fibration $\pi_{P,2}$.
\end{lemma}

\begin{remark}\label{rem:Venereau2}
As a corollary, we recover the fact that the second V\'en\'ereau polynomial is a  $\k[x]$-variable of $\k[x,y,z,u]$ when $\mathrm{char}(\k)\not=2$.  
\end{remark}

\begin{proof}[Proof of Lemma \ref{Lemm:V2}]
We apply Proposition~\ref{Prop:NewBiV} to the bivariable $\omega=ax+b^2y+bP(x)$ of Example~\ref{Ex:m=1}, which is associated with the data $\tau_a=\frac{y}{a}+\frac{P(x)}{ab}-\frac{1}{ab}P(\frac{\omega}{a})$, $\tau_b=-\frac{x}{b^2}$ and $f(x)=\frac{x}{ab^2}-\frac{1}{ab}P(\frac{x}{a})$. By Proposition~\ref{Prop:NewBiV}, the polynomial \[\hat{\omega}=\omega+aQ(a^m\tau_a)\] is a bivariable for every $m>\deg(P)$ and every $Q\in \k[a,b][x]$. In particular, we may choose    $m=3$ and $Q=\frac{ a x}{2c}$.  With these choices, we obtain the bivariable
\[\hat{\omega}=\omega+\frac{a^5}{2c}\tau_a=\omega+\frac{a^5}{2c}(\frac{y}{a}+\frac{P(x)}{ab}-\frac{1}{ab}P(\frac{\omega}{a})).\] 
Now, it remains to check that this bivariable indeed corresponds to the  transition function of $\pi_{P,2}$, i.e.~that it is associated with the function \[\hat{f}(x)=\frac{x}{ab^2}-\frac{1}{ab}P(\frac{x}{a})-\frac{a}{b^2}xP(\frac{x}{a})\]
that we computed at Example~\ref{Ex:f(x)_for_m=1,2}.

 As $(\omega,\tau_a)\in G_a$, we also have $(\hat\omega,\tau_a)\in G_a$ and we can let $\hat{\tau_a}=\tau_a$. To conclude the proof, we only need to prove that 
 \[ \hat{\tau}_b=\hat{\tau_a}-\hat{f}(\hat\omega)=\tau_a-\hat{f}(\hat\omega)=\tau_b+f(\omega)-\hat{f}(\hat\omega)\]
 is an element of the ring $\kk[a,b^{\pm1},x,y]$, i.e.~that it has no denominators in $a$. So, we need to prove that $f(\omega)-\hat{f}(\hat\omega)\in\kk[a,b^{\pm1},x,y]$. If we define $f_b=a^3f$  and $\hat{f_b}=a^3\hat{f}$, then  $f_b, \hat{f_b}\in \k[a,b^{\pm1}][x]$ and  we must then prove that 
 \begin{equation}\label{ffhat1}\hat{f_b}(\hat{\omega})\equiv_3 f_b(\omega),\tag{$\diamondsuit$}\end{equation} where $\equiv_i$ denotes the equality in $ \k[a,b^{\pm1},x,y]$ modulo $a^i$. As the element $\Delta=\frac{a^5}{2c}\tau_a$ belongs to the ideal $a^2\k[a,b^{\pm 1},x,y]$, we have that $\Delta^2\equiv_4 0$. By a Taylor expansion, we then obtain that 
 \[ \hat{f_b}(\hat{w})=\hat{f_b}(\omega+\Delta)\equiv_4\hat{f_b}(\omega)+\hat{f_b}'(\omega)\cdot \Delta.\]
 Since $\hat{f}(x)=f(x)-\frac{a}{b^2}xP(\frac{x}{a})$, we have that  $\hat{f_b}(\omega)\equiv_3 f_b(\omega)-\frac{a^4}{b^2}\omega P(\frac{\omega}{a})$. On the other hand, we have that $\hat{f_b}'(x)\equiv_1 -\frac{a}{b}P'(\frac{x}{a})$.
Therefore, we get that
 \[\begin{array}{rcl} \hat{f_b}(\hat{w})\equiv_3 f_b(\omega)-\frac{a^4}{b^2}\omega P(\frac{\omega}{a})-\frac{a}{b}P'(\frac{\omega}{a})\cdot \Delta.\end{array}\]
Finally, the result follows from the congruences $a^4P(\frac{\omega}{a})\equiv_3 a^2c\omega^2$, $aP'(\frac{\omega}{a})\equiv_1 2c\omega$ and $\Delta\equiv_3 -\frac{a^2}{2b}\omega^2$.
\end{proof}

 Lemma~\ref{lemme:variables_de_longueur2} does not only allow us to construct bivariables, but it is also useful to study $\A^2$-bundles $\rho_f\colon X_f\to \A^2_*$, $f\in\kk[a^{\pm1},b^{\pm1}][x]$, that are not necessarily trivial. More precisely, given a bundle $\rho_f$, one can construct other bundles $\rho_g$, that are isomorphic to $\rho_f$ as follows.  
 
\begin{proposition}\label{prop:procedure_transitions}
Let $f_b,g_b\in\kk[a,b^{\pm1}][x]$. If there exist $m\geq1$ and $Q\in\kk[a,b][x]$ such that \[g_b(x+aQ(f_b(x)))\equiv f_b(x)\pmod{a^m},\]
 then $\rho_{f_b/a^m}$ and $\rho_{g_b/a^m}$ are isomorphic $\A^2$-bundles.
\end{proposition}
 
\begin{proof}Define $v=x+aQ(a^my+f_b(x))$. By  Lemma~\ref{lemme:variables_de_longueur2}, $v$ is a $\k[a,b^{\pm1}]$-variable of $\k[a,b^{\pm 1},x,y]$. We also define $\alpha=(x-aQ(a^my),y)\in G_a$ and  $\gamma=(v,y+\frac{f_b(x)-g_b(v)}{a^m})\in G_{ab}$. Since $\alpha\circ (x,y+\frac{g_b(x)}{a^m})\circ \gamma=(x,y+\frac{f_b(x)}{a^m})$, the result will follow from Lemma~\ref{Lem:isomorphic-A2bundles} if we prove that $\gamma$  is actually an element of $ G_b$. As $(x,y+f_b(x)/a^m)\in G_{ab}$ has  Jacobian $1$, the same holds for $(v,y+f_b(x)/a^m)\in G_{ab}$ and for $\gamma\in G_{ab}$. To obtain that $\gamma\in G_b$ as desired, it suffices to observe that  $y+\frac{f_b(x)-g_b(v)}{a^m}\in \k[a^{\pm 1},b][x,y]$, as $g_b(v)\equiv g_b(x+aQ(f_b(x)))\equiv f_b(x)\pmod{a^m}$.
\end{proof} 
 

Although we unfortunately did not succeed to address the case of the first V\'en\'erau polynomial with our techniques, Proposition~\ref{prop:procedure_transitions} does lead to the following three results. Recall that the first  V\'en\'erau polynomial is a $\k[x]$-variable  if and only if the $\A^2$-bundle $\rho_{f_1}\colon X_{f_1}\to \A^2_*$ is isomorphic to the trivial bundle, where
\[f_1(x)=\frac{x}{ab^2}-\frac{x^2}{a^3b}-\frac{x^3}{a^2b^2}-\frac{x^4}{ab^3}.\]
 
 \begin{example}\label{Example:almost_V1}
 Suppose that $\mathrm{char}(\k)\not=2$. Then, one can apply Proposition~\ref{prop:procedure_transitions} with $Q(x)=\frac{x}{2}$ to show that   the bundle $\rho_g$ where
 \[g(x)=\frac{x}{ab^2}-\frac{x^2}{a^3b}-\frac{x^3}{a^2b^2}-\frac{5}{4}\cdot\frac{x^4}{ab^3}\]
 is isomorphic to the bundle $\rho_f$ where 
 \[f(x)=f_3(x)=\frac{x}{ab^2}-\frac{x^2}{a^3b}.\]
 Indeed, writing $f_b=a^3f,g_b=a^3g\in \k[a,b^{\pm1}][x]$, a straightforward calculation gives $g_b(x+aQ(f_b(x)))\equiv f_b(x)\pmod{a^3}$.
 
 As $\rho_{f_3}$ is a trivial bundle, the same holds for $\rho_g$. Note that $g$ differs from $f_1$ only by the coefficient $\frac{5}{4}$ in its last summand. 
 \end{example}
 
\begin{example}\label{Example:simplify_V1degree3}Assume that $\mathrm{char}(\k)\not=2$ and that $5$ is a square in $\k$. Then, one can apply Proposition~\ref{prop:procedure_transitions} with $Q(x)=\frac{1+\xi}{2}x$ where $\xi=\frac{1}{\sqrt{5}}$ to show that  the $\A^2$-bundles $\rho_{f_1}$ and $\rho_g$ are isomorphic, where \[g(x)=\frac{x}{ab^2}-\frac{x^2}{a^3b}+\frac{\xi x^3}{a^2b^2}\]
 is of degree three in $x$. 
 Indeed, writing $f_b=a^3f_1,g_b=a^3g\in \k[a,b^{\pm1}][x]$, a straightforward calculation gives $f_b(x+aQ(g_b(x)))\equiv g_b(x)\pmod{a^3}$.
\end{example}

\begin{example}\label{Example:simplify_V1degree4}Assume that $\mathrm{char}(\k)\not=2$. Then, one can apply Proposition~\ref{prop:procedure_transitions} with $Q(x)=\frac{x}{2}$ to show   that the $\A^2$-bundles $\rho_{f_1}$ and $\rho_g$ are isomorphic, where
 \[g(x)=\frac{x}{ab^2}-\frac{x^2}{a^3b}+\frac{1}{4}\frac{x^4}{ab^3}\]
 has the same degree than $f_1$ but with one summand  less. 
 Indeed, writing $f_b=a^3f_1,g_b=a^3g\in \k[a,b^{\pm1}][x]$, a straightforward calculation gives $f_b(x+aQ(g_b(x)))\equiv g_b(x)\pmod{a^3}$.
\end{example}
 
\section{Local triviality on the blow-up}\label{Sec:BlowUp}
 
By Theorem~\ref{Theorem:fcs_de_transition}, every map 
\[\begin{array}{rrcl}
\pi_{P,n}\colon& \A^4& \to & \A^2\\
&(x,y,z,u) & \mapsto & (x,v_{P,n})\end{array}\]
has the structure of an $\A^2$-bundle over the punctured affine plane. In fact, a stronger fact holds: every map $\pi_{P,n}$ yields an $\A^2$-bundle over $\A^2$ blown-up at the origin.

\begin{theorem}\label{Thm:BlowUp}
If we denote by $\epsilon\colon \hat\A^2\to \A^2$ the blow-up of $(0,0)\in \A^2$ and by $\eta\colon \hat\A^4\to \A^2$ the blow-up of $x=y=0$, then the pull-back of $\pi_{P,n}$ by $\epsilon$
\[\hat{\pi}_{P,n}\colon\A^4 \times_{\A^2} \hat\A^2\to \hat\A^2\]
is a $($locally trivial$)$ $\A^2$-bundle. Moreover, $\A^4 \times_{\A^2} \hat\A^2\to \A^2$ is simply the blow-up of $x=y=0$.
\end{theorem}

\begin{proof}
The blow-up $\epsilon\colon \hat\A^2\to \A^2$ of the origin can be seen at the projection $((a,b),[A:B])\mapsto (a,b)$, where
\[\hat{\A}^2=\{((a,b),[A:B])\in \A^2\times \mathbb{P}^1\mid aB=bA\}.\]
We consider the two open subsets defined by $A\not=0$ and by $B\not=0$. They correspond to $\Spec(\kk[a,\frac{b}{a}])$ and $\Spec(\kk[b,\frac{a}{b}])$ respectively, and their intersection is isomorphic to $\Spec(\kk[a,b,\frac{a}{b},\frac{b}{a}])$. The transition function $\alpha^{-1}\circ (x,y+f(x))\circ\beta\in G_{ab}$ computed in Lemma~\ref{Lemm:TransitionFuncabba}\ref{Fcmm} is then an isomorphism over this intersection. This proves that $\hat{\pi}_{P,n}$ is a locally trivial $\A^2$-bundle. Moreover, as $\hat\A^2\to \hat\A^2$ is the blow-up of $(0,0)$, the morphism $\A^4 \times_{\A^2} \hat\A^2\to \A^2$ is the blow-up of $(\pi_{P,n})^{-1}(0,0)$, which is the surface given by $x=y=0$.
\end{proof}

\begin{lemma}\label{Lemm:TransitionFuncabba} 
Let $n\ge 1$, $P\in \k[z]$ and  let $\omega\in \k[a,b,x,y]$ be defined by $\omega=ax+b^2y+bP(x)$. For every $m\ge 1$, define  \[f_m=\frac{x}{ab^2}-\frac{1}{ab^m}\cdot \frac{b^m-(a^nx)^m}{b-a^nx}P(\frac{x}{a})\in \kk[a^{\pm1},b^{\pm1}][x].\] Then, the following hold:
\begin{enumerate}
\item\label{Fcm1}
The elements $\alpha=(\omega,\dfrac{y}{a}+\dfrac{P(x)-P(a^{-1}\omega)}{ab})$ 
 and $\beta=(\omega,-b^{-2}x)$ belong to the groups $G_a$ and to $G_b$, respectively, and they satisfy that $\alpha^{-1}\circ (x,y+f_1(x))\circ\beta=(x,y).$
 \item\label{Fcmm}
 For each $m\ge 1$, the  components of $\alpha^{-1}\circ (x,y+f_m(x))\circ\beta\in G_{ab}$, as well as the  components of its inverse, all belong  to the ring $\kk[a,b,\frac{a}{b},\frac{b}{a},x,y]$.
\end{enumerate}
 \end{lemma}
\begin{proof}
Assertion~\ref{Fcm1} is straightforward to check. Let us denote $T_m=(x,y+f_m(x))\in G_{ab}$ for each $m\ge 1$. In particular, $T_1=(x,y+\frac{x}{ab^2}-\frac{1}{ab}P(\frac{x}{a}))\in G_{ab}$ and we observe that 
\begin{align*}
(T_1)^{-1}\circ T_m&=\Big(x,y+f_m(x)-\frac{x}{ab^2}+\frac{1}{ab}P(\frac{x}{a})\Big)\\
&=\Big(x,y+(\frac{1}{ab}-\frac{1}{ab^m}\cdot\frac{b^m-(a^nx)^m}{b-a^nx})P(\frac{x}{a})\Big)\\
&=\Big(x,y+(\frac{1}{ab}-\frac{1}{ab^m}\cdot\sum_{k=0}^{m-1}b^{m-1-k}(a^nx)^k)P(\frac{x}{a})\Big)\\
&=\Big(x,y-\frac{1}{ab}\sum_{k=1}^{m-1}(\frac{a^nx}{b})^k\cdot P(\frac{x}{a})\Big).
\end{align*}

Since, $\beta^{-1}=(-b^2y, \frac{x}{b^2}+ay-\frac{1}{b}\cdot P(-b^2y))$ and  $T_1=\alpha\circ\beta^{-1}$, we can now calculate  
\begin{align*}
\alpha^{-1}\circ T_m\circ\beta&=\beta^{-1}\circ(T_1)^{-1}\circ T_m\circ\beta\\
&=\beta^{-1}\circ\Big(x,y-\frac{1}{ab}\sum_{k=1}^{m-1}(\frac{a^nx}{b})^k\cdot P(\frac{x}{a})\Big)\circ\beta\\
&=\beta^{-1}\circ\Big(\omega,-\frac{x}{b^2}-\frac{1}{ab}\sum_{k=1}^{m-1}(\frac{a^n\omega}{b})^k\cdot P(\frac{\omega}{a})\Big)\\
&=\Big(x+\frac{b}{a}\sum_{k=1}^{m-1}(\frac{a^n\omega}{b})^k\cdot P(\frac{\omega}{a}), \frac{\omega}{b^2}-\frac{ax}{b^2}-\frac{1}{b}\sum_{k=1}^{m-1}(\frac{a^n\omega}{b})^k\cdot P(\frac{\omega}{a})\\
&\qquad\qquad\qquad\qquad -\frac{1}{b}P\left(x+\frac{b}{a}\sum_{k=1}^{m-1}(\frac{a^n\omega}{b})^k\cdot P(\frac{\omega}{a})\right)\Big)\\
&=\Big(x+\frac{b}{a}\sum_{k=1}^{m-1}(\frac{a^n\omega}{b})^k\cdot P(\frac{\omega}{a}), y+-\frac{1}{b}\sum_{k=1}^{m-1}(\frac{a^n\omega}{b})^k\cdot P(\frac{\omega}{a})\\
&\qquad\qquad\qquad+\frac{1}{b}P(x)-\frac{1}{b}P\left(x+\frac{b}{a}\sum_{k=1}^{m-1}(\frac{a^n\omega}{b})^k\cdot P(\frac{\omega}{a})\right)\Big)
\end{align*}
As $\frac{\omega}{a}\in \kk[a,b,\frac{a}{b},\frac{b}{a}]$, it is straightforward to check that the two components of the above map are contained in $\kk[a,b,\frac{a}{b},\frac{b}{a}]$. Moreover, since the Jacobian determinants of  $\alpha$, $\beta$ and $T_m$ are equal to $1$, the Jacobian determinant of $\alpha^{-1}\circ T_m\circ\beta$ is also equal to $1$ and the components of its inverse also belong to the same ring.
\end{proof}

\section{The varieties \texorpdfstring{$X_f$}{Xf} for small denominators} \label{Sec:Xfmn1}
 In this section, we study the varieties $X_f$ introduced in Definition~\ref{Defi:Xf}. One may always choose $m,n\ge 0$ such that $a^mb^nf(x)\in \k[a,b,x]$. Our main result is Proposition~\ref{Prop:Casemn1} which gives a criterion to decide whether the corresponding bundle $\rho_f\colon X_f\to \A^2_*$ is a trivial $\A^2$-bundle, in the special case where one of the two integers $m$ or $n$ is at most $1$.
 
We will proceed as follows. We will first realise $X_f$ as an open subset of an affine hypersurface $Y\subseteq \A^5$ (see Lemma~\ref{Lemm:HypersurfaceA5}). Then, we will study  when the morphism $\rho_f\colon X_f\to \A^2_*$ extends to a (locally) trivial $\A^2$-bundle $Y\to \A^2$ (see Lemma~\ref{lemm:piYtrivial}). Afterwards, we will compute the ring of regular functions on a variety $X_f$ (see Proposition~\ref{Prop:RingXf}) and show that, in the case where $m=1$ or $n=1$,  this ring is equal to the ring of regular functions  of $Y$. We will finally use this result to prove Proposition~\ref{Prop:Casemn1}. 


\begin{lemma}\label{Lemm:HypersurfaceA5}
For each $f=f(a,b,x)\in \kk[a^{\pm 1},b^{\pm 1}][x]$ and for all integers $m,n\ge 0$ such that $P(a,b,x)=a^mb^nf\in\k[a,b,x]$, the variety $X_f$ admits an open embedding into the hypersurface $Y\subset\A^5=\Spec(\kk[a,b,x,u,v])$ defined by the equation
\[a^{m}u-b^{n}v=P(a,b,x).\] 
More precisely, $X_f$ is isomorphic to $Y\setminus\{a=b=0\}$, via the isomorphisms
\[\begin{array}{l}
\varphi\colon U_a\times \A^2\to Y\backslash \{a=0\}, ((a,b), (x,y))
\mapsto (a,b,x,b^ny+a^{-m}P(a,b,x),a^my).\\
\psi\colon U_b\times \A^2\to Y\backslash \{b=0\}, ((a,b), (x,y))
\mapsto (a,b,x,b^ny,a^my-b^{-n}P(a,b,x)).\end{array}\]
\end{lemma}
\begin{proof}
We observe that the maps $\psi,\varphi$ are isomorphisms, whose  inverse maps are given by 
\[\begin{array}{l}
\varphi^{-1}\colon Y\setminus \{a=0\}\iso U_a\times \A^2, (a,b,x,u,v)\mapsto ((a,b), (x,a^{-m}v))\\
\psi^{-1}\colon Y\setminus \{b=0\}\iso U_a\times \A^2, (a,b,x,u,v)\mapsto ((a,b), (x,b^{-n}u)).\end{array}\]
Moreover, as we obtain the transition function of $X_f$ via the composition
\[\begin{array}{cccc}
\psi^{-1}\circ\varphi\colon& U_{ab}\times\A^2 &\to& U_{ab}\times\A^2\\
&((a,b), (x,y))&\mapsto &((a,b),(x,y+a^{-m}b^{-n}P(a,b,x))),
\end{array}\]  we see that the two maps $\varphi$ and $\psi$ induce inverse isomorphisms between $X_f$ and $Y\setminus\{a=b=0\}$.
\end{proof}

\begin{lemma}\label{lemm:piYtrivial}
Let $m,n\geq1$ and $P\in\kk[a,b,x]$. Define $Q=a^{m}u-b^{n}v-P\in \k[a,b,x,u,v]$ and let $\pi\colon\Spec(\k[a,b,x,u,v]/Q)\to \A^2$ be the morphism defined by $(a,b,x,u,v)\mapsto (a,b)$. Then, the following conditions are equivalent:
\begin{enumerate}
\item \label{PQtrivialA2bundle}
$\pi$ is a trivial $\A^2$-bundle.
\item \label{PQloctrivialA2bundle}
$\pi$ is a locally trivial $\A^2$-bundle.
\item\label{PQfibre0isA2}
The fibre $\pi^{-1}((0,0))$ is isomorphic to $\A^2$.
\item\label{PQPdeg1}
$P(0,0,x)\in \k[x]$ is of degree $1$.
\item\label{PQkabvariable}
$Q$ is a $\kk[a,b]$-variable of $\kk[a,b,x,u,v]$.
\end{enumerate}
\end{lemma}

\begin{proof}

$\ref{PQtrivialA2bundle}\Rightarrow\ref{PQloctrivialA2bundle}$: Because any trivial bundle is locally trivial.

$\ref{PQloctrivialA2bundle}\Rightarrow\ref{PQfibre0isA2}$: As $\pi$ is a locally trivial $\A^2$-bundle, the fibre over $(0,0)$ is isomorphic to $\A^2$.

$\ref{PQfibre0isA2}\Rightarrow\ref{PQPdeg1}$: The fibre over $(0,0)$ is defined by  the equation $P(0,0,x)$. If this fibre is isomorphic to $\A^2$, the polynomial $P(0,0,x)$ is then of degree $1$.

$\ref{PQPdeg1}\Rightarrow\ref{PQkabvariable}$:
We suppose that $P(0,0,x)$ is of degree $1$ and  prove that $Q$ is a $\k[a,b]$-variable of $\k[a,b,x,u,v]$. We may apply elements of $G=\Aut_{\k[a,b]}(\k[a,b,x,u,v])$, since these automorphisms send $\k[a,b]$-variables onto $\k[a,b]$-variables. Applying such an automorphism that sends  $x$ onto $\xi x+\mu$ for suitable  $\xi\in \k^*$ and $\mu\in \k$, we  can suppose that $P(0,0,x)=x$ and  replace $Q$ with
\[Q_1=x+a^mu+b^nv+aP_1(a,b,x)+bP_2(a,b,x)\in\kk[a,b,x,u,v],\]
where $P_1,P_2\in \k[a,b,x]$. Now, Lemma~\ref{lemme:variables_de_longueur2}  implies that $x+a^mu+aP_1(a,b,x)$ is a $\kk[a,b,v]$-variable of $\kk[a,b,v][x,u]$. There is thus $g\in G$ such that $v=g(v)$ and $x=g(x+a^mu+aP_1(a,b,x))$. Hence, $Q_2=g(Q_1)$ is of the form
\[Q_2=x+b^nv+bP_3(a,b,x,u)\in\kk[a,b,x,u,v],\]
where $P_3\in \k[a,b,x,u]$.
 In turn, again by Lemma~\ref{lemme:variables_de_longueur2}, the element $Q_2$ is a $\kk[a,b,u]$-variable of $\kk[a,b,u][x,v]$ and  it is thus in particular  a $\kk[a,b]$-variable.

$\ref{PQkabvariable}\Rightarrow\ref{PQtrivialA2bundle}$: If $Q$ is a $\kk[a,b]$-variable of $\kk[a,b,x,u,v]$, there exists a $\kk[a,b]$-automorphism of $\kk[a,b,x,u,v]$ which sends $Q$ onto $x$. This isomorphism trivialises the $\A^2$-bundle $\pi$.
\end{proof} 
\begin{proposition}\label{Prop:RingXf}For each $f=f(a,b,x)\in \kk[a^{\pm 1},b^{\pm 1}][x]$, the following hold:
\begin{enumerate}
\item\label{RegXf}
The ring $\mathcal{O}(X_f)$ of regular functions on $X_f$ is given on the two charts $U_a \times \A^2$ and $U_{b}\times \A^2$  by $R_a$ and $R_b$, respectively, where
\[R_a=\kk[a^{\pm 1},b,x,y]\cap \kk[a,b^{\pm 1},x,y+f],\]
\[R_b=\kk[a,b^{\pm 1},x,y]\cap \kk[a^{\pm 1},b,x,y-f].\]
\item\label{RegXfInclusion} For all integers $m,n\ge 0$ such that $P=a^mb^nf\in\k[a,b,x]$, we have
\[\kk[a,b,x,a^my,b^ny+a^{-m}P]\subseteq R_a\text{ and }\kk[a,b,x,b^ny,a^my-b^{-n}P]\subseteq R_b.\]
Moreover, the first inclusion is an equality if and only if the second is also an equality.
\item\label{RgXfequalities1}
If $1\in \{m,n\}$ and $P(0,0,x)\not=0$, the inclusions in \ref{RegXfInclusion} are equalities.
\item\label{RgXwhenEq}
Under the assumption that  the inclusions in \ref{RegXfInclusion} are equalities, we have that $X_f$ is a trivial $\A^1$-bundle if and only if $P(0,0,x)=P|_{a=b=0}$ is of degree $1$.
\end{enumerate}
\end{proposition} 
 
\begin{proof}Denote by $\psi\colon U_a \times \A^2\dasharrow U_b\times \A^2$ the birational map $((a,b),(x,y))\mapsto ((a,b),(x,y+f(a,b,x))$ and by $\varphi$ its inverse.

\ref{RegXf}: A regular function on $X_f$ is a function that is regular on $U_a \times \A^2$ and $U_{b}\times \A^2$.
Regular functions  on $U_a\times \A^2$ correspond to the ring $\mathcal{O}(U_a\times \A^2)\cap \psi^*(\mathcal{O}(U_b\times \A^2))$. Similarly, regular functions  on $U_b\times \A^2$ correspond to $\mathcal{O}(U_b\times \A^2)\cap \varphi^*(\mathcal{O}(U_a\times \A^2))$.
 As $\mathcal{O}(U_a\times \A^2)=\kk[a^{\pm 1},b,x,y]$ and $\mathcal{O}(U_b\times \A^2)=\kk[a,b^{\pm 1},x,y]$, Assertion \ref{RegXf} follows.


\ref{RegXfInclusion} As $P=a^mb^nf\in \k[a,b,x]$, we have $b^{-n}P=a^mf\in \k[a,b^{\pm1}, x]$ and $a^{-m}P=b^mf\in \k[a^{\pm1},b, x]$. Hence, $b^n(y+f)=b^ny+a^{-m}P\in R_a$ and $a^m(y-f)=a^my-b^{-n}P\in R_b$. This shows that the two inclusions in Assertion \ref{RegXfInclusion} hold. Moreover, since $\varphi^*(b^ny+a^{-m}P)=b^ny$ and $\psi^*( a^my-b^{-n}P)=a^my$, it follows these inclusions are either both strict or both an equality.

\ref{RgXfequalities1}:
We now assume that $m=1$ and $P(0,0,x)\not=0$. This implies that $P(0,b,x)\in \k[b,x]\setminus b\k[b,x]$. We want to prove that $R_a\subseteq \kk[a,b,x,ay,b^ny+a^{-1}P]$. Writing $P=P(0,b,x)+aS$ for some suitable $S\in \k[a,b,x]$, we obtain that $R_a=\kk[a^{\pm 1},b,x,y]\cap \kk[a,b^{\pm 1},x,y+\frac{P_0}{ab^n}]$ and $\kk[a,b,x,ay,b^ny+a^{-1}P]=\kk[a,b,x,ay,b^ny+a^{-1}P_0]$. So, we may assume that $P=P(0,b,x) \in \k[b,x]\setminus b\k[b,x]$. 

Let us replace $y$ by $\frac{y}{a}$  and define $u=\frac{b^ny+P}{a}\in \k[a^{\pm 1},b,x,y]$. Doing so, we now need to prove that any element 
$w\in \k[a^{\pm 1},b,x,y]\cap \k[a,b^{\pm 1},x,u]$ belongs to the ring $\k[a,b,x,y,u]$. 

We write $w=\frac{W}{b^s}$ with $W\in \k[a,b,x,u]\subseteq \k[a,b,x,y,u]$ and $s\ge 0$. Using  the fact that $au=b^ny+P$, we can rewrite $W$ as
\[W=\sum_{i=1}^d u^i Q_i +R,\]
where $d\ge 1$, $Q_i\in \k[b,x,y]$ for each $i$ and $R\in \k[a,b,x,y]$.

If $s=0$, then $w=W\in \k[a,b,x,y,u]$ and we have nothing to prove. Therefore,   we  assume that $s>0$. This implies that $W\equiv 0\pmod{b}$, since  $w=\frac{W}{b^s}\in \k[a^{\pm 1},b,x,y]$. Denoting by $\hat{P}\in \k[x],\hat{Q}_i\in \k[x,y], \hat{R}\in \k[a,x,y]$ the elements such that $\hat P\equiv P,\hat Q_i\equiv Q_i,\hat R\equiv R\pmod b$, gives us the equality 
$
0=\sum_{i=1}^d (\frac{\hat P}{a})^i \hat Q_i +\hat{R}.$ Multiplying this equality by $a^d$,   we then obtain that
\[\begin{array}{l}0=\sum_{i=1}^d \hat{P}^ia^{d-i} \hat Q_i +a^d\hat{R}=\sum_{i=0}^{d-1} \hat{P}^{d-i}\hat Q_{d-i}a^{i} +a^d\hat{R}\in \k[a,x,y]=\k[x,y][a].\end{array}\]
Note that $\hat{P}\not=0$, since we assumed that $P\not\in b\cdot \k[b,x]$. Therefore, since all coefficients of $a^i$ with $i<d$ are equal to $0$, we   get $\hat Q_i=0$ for all $i=1,\dots,d$, and that $\hat{R}=0$. 

This means that $W$ is divisible by $b$ and we can thus write $w=\frac{W}{b^s}=\frac{W'}{b^{s-1}}$ with $W'\in \k[a,b,x,y,u]$. Arguing as before, we eventually conclude that $w\in \k[a,b,x,y,u]$, as desired.

The case where $n=1$ and $P(0,0,x)\not=0$  is similar, when exchanging the roles of $a$ and $b$.

\ref{RgXwhenEq}: We now assume that the inclusions of \ref{RegXfInclusion} are equalities, and prove that $X_f$ is a trivial $\A^1$-bundle if and only if $P(0,0,x)=P|_{a=b=0}$ is of degree~$1$. 

The equalities of  \ref{RegXfInclusion} imply that the ring $\mathcal{O}(X_f)$ of regular functions on $X_f$ is given  on the two charts $U_a \times \A^2$ and $U_{b}\times \A^2$ by
\[\kk[a,b,x,a^my,b^ny+a^{-m}P]\text{ and }\kk[a,b,x,b^ny,a^my-b^{-n}P],\]
respectively. In particular, considering the open embedding $X_f\hookrightarrow Y$ defined in Lemma~\ref{Lemm:HypersurfaceA5}, the ring of regular functions on $X_f$ is the restriction of the ring of regular functions on $Y$. Moreover, the $\A^2$-bundle $\rho_f\colon X_f\to \A^2_*$ is the restriction of $\pi\colon Y\to \A^2$ given by $(a,b,x,u,v)\to (a,b)$ (see Lemma~\ref{Lemm:HypersurfaceA5}). 

If $P(0,0,x)$ is of degree $1$, then the morphism $\pi\colon Y\to \A^2$, is a trivial $\A^2$-bundle (Lemma~\ref{lemm:piYtrivial}), so the restriction $\rho_f$ is also a trivial $\A^2$-bundle.

Conversely, we now assume that $\rho_f$ is a trivial $\A^2$-bundle and prove that $P(0,0,x)$ is of degree $1$. Let  $\chi\colon  \A^2\times \A^2_*\iso X_f$ be an isomorphism such that  $\rho_f\circ \chi$ is the projection onto the second factor. Then $\chi$ extends to a birational map $\A^4\dasharrow Y\subseteq \A^5$ between two affine varieties, which induces an isomorphism between their regular rings, since $\mathcal{O}(X_f)=\mathcal{O}(Y)|_{X_f}$ and $\mathcal{O}( \A^2\times \A^2_*)=\mathcal{O}(\A^4)|_{\A^2\times \A^2_*}$. Thus, $\chi$ is in fact an isomorphism $\A^2\iso Y$. By Lemma~\ref{Lemm:HypersurfaceA5}, it sends $\A^2\times (0,0)$ onto $Y\setminus X_f=Y\cap \{a=b=0\}=\pi^{-1}(0,0)$. Since $\pi^{-1}(0,0)\simeq \A^2$, we can now conclude, by Lemma~\ref{lemm:piYtrivial}, that the polynomial $P(0,0,x)$ is of degree $1$.  
\end{proof}

\begin{proposition}\label{Prop:Casemn1}
Let $f=f(a,b,x)\in \kk[a^{\pm 1},b^{\pm 1}][x]$ and let $m,n\ge 0$ be   such that $P(a,b,x)=a^mb^nf\in\k[a,b,x]$.
\begin{enumerate}
\item\label{Propmn0}
If $m=0$ or $n=0$, then the $\A^2$-bundle $\rho_f\colon X_f\to\A^2_*$ is  trivial.
\item\label{Propmn1}
Suppose that $P(0,0,x)\not=0$, i.e.~ that $m$ and $n$ are minimal with the property that $P(a,b,x)=a^mb^nf\in\k[a,b,x]$. If $m=1$ or $n=1$,then the $\A^2$-bundle $\rho_f\colon X_f\to\A^2_*$ is trivial if and only if $P(0,0,x)$ is a polynomial in $\kk[x]$ of degree one.
\end{enumerate}
\end{proposition}
\begin{proof}
\ref{Propmn0}: If $m=0$ or $n=0$, then $f\in \k[a,b^{\pm 1}][x]$ or $f\in \k[a^{\pm 1},b][x]$. This implies that the corresponding $\A^1$-bundle $X_f\to \A^2_{*}\times \A^1$ is trivial  (see Lemma~\ref{Lemm:A1bundle}). In particular, $\rho_f\colon X_f\to \A^2_*$ is a trivial $\A^2$-bundle.

\ref{Propmn1}: This follows directly from the  assertions \ref{RgXfequalities1} and \ref{RgXwhenEq} in Proposition~\ref{Prop:RingXf}.
\end{proof} 
 
 \begin{remark}
If $m,n\geq2$, then the $\A^2$-bundle $\rho_f\colon X_f\to\A^2_*$ can be trivial, even if the polynomial $P(0,0,x)$ is not of degree one. This occurs for example in the special case where $f(a,b,x)=a^{-3}b^{-2}(a^2x-bx^2)$, which corresponds to the third V\'en\'ereau polynomial. The next example is another instance of this phenomena when $n=m=2$.
\end{remark}

\begin{example}
Let $P(a,b,x)=(a+b)x$ and $f(a,b,x)=a^{-2}b^{-2}P(a,b,x)=a^{-2}b^{-2}(a+b)x$. Then, the $\A^2$-bundle $\rho_f\colon X_f\to\A^2_*$ corresponds to the bivariable $\omega=a^2x+(b-a)y$ and is thus trivial. Indeed, a straightforward calculation gives $\alpha\circ\beta^{-1}=(x,y+f(a,b,x))$ where 
\[\alpha=(\omega,a^{-2}y)\in G_a\text{ and }\beta=(\omega, b^{-2}y-b^{-2}(a+b)x)\in G_b.\]
\end{example}
 

\
 
\end{document}